\documentclass[11pt]{article}

\usepackage{amsfonts}
\usepackage{amsmath}
\usepackage{amssymb}
\usepackage{amsthm}

\usepackage[UKenglish]{babel}
\usepackage{enumerate}
\usepackage{footnote}
\usepackage{hyperref}
\usepackage{latexsym}
\usepackage{lmodern}
\theoremstyle{plain}
\newtheorem{theorem}{Theorem}

\newtheorem{cor}[theorem]{Corollary}
\newtheorem{lemma}[theorem]{Lemma}
\theoremstyle{remark}

\usepackage{caption}
\usepackage{subcaption}
\DeclareCaptionType{algorithm}
\newcommand{\sset}[1]{\left\{#1\right\}}
\newcommand{\ceil}[1]{\left\lceil#1\right\rceil}

\usepackage[margin=2.25cm]{geometry}
\usepackage{marvosym}

\title{Approximately coloring graphs without long induced paths}
\author{Maria Chudnovsky \thanks{Supported by NSF grant DMS-1550991.}
\\
Princeton University, Princeton, NJ 08544
\\
\\
Oliver Schaudt
\\
Universit\"at zu K\"oln, K\"oln, Germany
\\
\\
Sophie Spirkl
\\
Princeton University, Princeton, NJ 08544
\\
\\
Maya Stein \thanks{Supported by Fondecyt grant 1140766 and Millennium Nucleus Information and Coordination in Networks.}
\\
Universidad de Chile, Santiago, Chile
\\
\\
Mingxian Zhong
\\
Columbia University, New York, NY 10027}
\date{\today}

\begin{document}

\maketitle

\begin{abstract}
It is an open problem whether the 3-coloring problem can be solved in polynomial time in the class of graphs that do not contain an induced path on $t$ vertices, for fixed $t$.
We propose an algorithm that, given a 3-colorable graph without an induced path on $t$ vertices, computes a coloring with $\max\sset{5,2\ceil{\frac{t-1}{2}}-2}$ many colors.
If the input graph is triangle-free, we only need $\max\sset{4,\ceil{\frac{t-1}{2}}+1}$ many colors.
The running time of our algorithm is $O((3^{t-2}+t^2)m+n)$ if the input graph has $n$ vertices and $m$ edges.
\end{abstract}

\section{Introduction}

A \emph{$k$-coloring} of a graph $G$ is a function $c : V(G) \rightarrow \sset{1,\dots, k}$ so that $c(v) \neq c(u)$ for all $vu \in E(G)$. 
In the \emph{$k$-coloring problem}, one has to decide whether a given graph admits a $k$-coloring or not; it is NP-complete for all $k \geq 3$, as Karp proved in his seminal paper~\cite{karp}.

\paragraph{Coloring $H$-free graphs.}
One way of dealing with this hardness is to restrict the structure of the instances.
In this paper we study \emph{$H$-free graphs}, that is, graphs that do not contain a fixed graph $H$ as an induced subgraph.
It is known that the $k$-coloring problem is NP-hard on $H$-free graphs if $H$ is any graph other than a subgraph of a chordless path~\cite{Hol81,KL07,KKTW01,LG83}.
Therefore, we further restrict our attention to $P_t$-free graphs, $P_t$ being the chordless path on $t$ vertices.

A substantial number of papers study the complexity of coloring $P_t$-free graphs, and most of the results are gathered in the survey paper of Golovach \emph{et al.}~\cite{GJPS15}.
Let us recall a few results that define the current state-of-the-art regarding the complexity of $k$-coloring in $P_t$-free graphs.

\begin{theorem}[Bonomo \emph{et al.}~\cite{BCMSSZ15}] The 3-coloring problem can be solved in polynomial time in the class of $P_7$-free graphs.
This holds true even if each vertex comes with a subset of $\{1,2,3\}$ of feasible colors.
\end{theorem}

It is an intriguing open question whether the 3-coloring problem is solvable in polynomial time in the class of $P_t$-free graphs, whenever $t>7$ is fixed.

Going back to $P_5$-free graphs, an elegant algorithm of Ho\`ang \emph{et al.}~\cite{hoang} shows that this class is structurally restricted enough to allow for a polynomial time algorithm solving the $k$-coloring problem.

\begin{theorem}[Ho\`ang \emph{et al.}~\cite{hoang}] The $k$-coloring problem can be solved in polynomial time in the class of $P_5$-free graphs, for each fixed $k$.
\end{theorem}

The above result is interesting also for the fact that if $k$ is part of the input, the $k$-coloring problem in $P_5$-free graphs becomes NP-hard again~\cite{KKTW01}.
Regarding negative results, the following theorem of Huang is the best known so far.

\begin{theorem}[Huang~\cite{huang}] For all $k \geq 5$, the $k$-coloring problem is NP-complete in the class of $P_6$-free graphs.
Moreover, the 4-coloring problem is NP-complete in the class of $P_7$-free graphs.
\end{theorem}

The only cases when the complexity of $k$-coloring $P_t$-free graphs is not known is when $k=4$, $t=6$, or when $k=3$ and $t \geq 8$. 
Our contribution is an approximation algorithm for the latter case. This line of research was first suggested to us by Chuzhoy~\cite{chuzhoy}. 

\paragraph{Approximation.}
The hardness of approximating the $k$-coloring problem has been in the focus of the research on approximation algorithms.
Dinur, Mossel and Regev \cite{IMR09} proved that coloring a 3-colorable graph with $C$ colors, where $C$ is any constant, is NP-hard assuming a variant of the Unique Games Conjecture. More precisely, the assumption is that a certain label cover problem is NP-hard (where the label cover instances are what the authors call $\alpha$-shaped).

On the upside, it is known how to color a 3-colorable graph with relatively few colors in polynomial time, and there has been a long line of subsequent improvements on the number of colors needed. 
The current state of the art, according to our knowledge, is the the following result, which combines a semidefinite programming result by Chlamtac~\cite{Chl07} with a combinatorial algorithm for the case of large minimum degree.

\begin{theorem}[Kawarabayashi and Thorup~\cite{kawa}] There is a polynomial time algorithm to color a 3-colorable $n$-vertex graph with $O(n^{0.19996})$ colors.
\end{theorem}

In this work we combine these two lines of research and strive to use the structure of $P_t$-free graphs to give an approximation algorithm for the 3-coloring problem.
We are inspired by a result of Gy\'arf\'as, who proved the following.
\begin{theorem}[Gy\'arf\'as~\cite{gyarfas}] \label{thm:gyarfas} If $G$ is a graph with no induced subgraph isomorphic to $P_t$, then $\chi(G) \leq (t-1)^{\omega(G)-1}$. 
\end{theorem}
Thus, for a graph with no $P_t$, we can check if it is $(t-1)^2$-colorable or not 3-colorable by checking whether it contains a $K_4$. For a connected graph, Theorem~\ref{thm:gyarfas} also holds if the requirement of being $P_t$-free is weakened to the assumption that there is a vertex $v$ in $G$ that does not start an induced $P_t$ in $G$. 
We use a technique similar to the proof of Theorem~\ref{thm:gyarfas} in the proof of our key lemma, Lemma~\ref{lem:one}. 
We take advantage, however, from the fact that the input graph is 3-colorable.
This allows us to improve the bound of $(t-1)^2$ on the number of colors given by Gy\'arf\'as' theorem.
We remark that our result is not an improvement of Theorem~\ref{thm:gyarfas}, but incomparable to it.

\paragraph{Our contribution.}
We prove the following.
\begin{theorem} \label{thm:main}
Let $t \in \mathbb{N}$.
There is an algorithm that computes for any 3-colorable $P_t$-free graph $G$
\begin{enumerate}[(a)]
\item a coloring of $G$ with at most $\max\sset{5,2\ceil{\frac{t-1}{2}}-2}$ colors, and a triangle of $G$, or 
\item a coloring of $G$ with at most $\max\sset{4,\ceil{\frac{t-1}{2}}+1}$ colors
\end{enumerate}
with running time $O((3^{t-2}+t^2)|E(G)|+|V(G)|)$.
\end{theorem}
There is a variant of this problem where we replace the requirement that $G$ is $P_t$-free with the weaker restriction that $G$ has at least one vertex which is not a starting vertex of a $P_t$ in each connected component. We give an algorithm for this harder problem as well, with a worse approximation bound, see Lemmas~\ref{lem:one} and~\ref{lem:rt1} below. Additionally, we give a hardness result, Theorem~\ref{lem:hardness}, to show that Lemma~\ref{lem:one} can probably not be improved.

We remark that our algorithm can easily be implemented so that it
takes an arbitrary graph as its input.  It then either refutes the
graph by outputting that it contains a $P_t$ or that it is not
3-colorable or computes a coloring as promised by
Theorem~\ref{thm:main}.  In the case that the graph is refuted for not being
3-colorable, the algorithm can output a certificate that is easily
checked in polynomial time. If the graph is refuted because it
contains an induced $P_t$, our algorithm outputs the path.


\section{Algorithm}

We start with a lemma that uses ideas from Theorem~\ref{thm:gyarfas} to color connected graphs in which some vertex does not start a $P_t$. It is exact up to $t = 4$; in Section~\ref{sec:hard} we show that 3-coloring becomes NP-hard for $t \geq 5$, which means that our result is tight in this sense. 

\begin{lemma} \label{lem:one} Let $G$ be connected, $v \in V(G)$, and $t \in \mathbb{N}$. There is a polynomial-time algorithm that outputs
\begin{enumerate}[(a)]
\item \label{it:not3} that $G$ is not 3-colorable, or
\item \label{it:path} an induced path $P_t$ starting with vertex $v$, or
\item \label{it:ccc} a $\max\sset{2, t-2}$-coloring of $G$, or
\item \label{it:ddd} a $\max\sset{3, 2t-5}$-coloring of $G$ and a triangle in $G$.
\end{enumerate}
\end{lemma}

\begin{proof}
We prove this by induction on $t$. 
For  $t \leq 4$, let $Z = V(G) \setminus (\sset{v} \cup N(v))$. Consider a component~$C$ of $G[Z]$. By connectivity, there is a vertex $x\in V(C)$ such that $N(x)\cap N(v)\neq\emptyset$. Since $G$ has no $P_4$ starting at $v$, each neighbor of $x$ in $C$ is adjacent to all of $N(x)\cap N(v)$. Thus $N(y)\cap N(v)=N(x)\cap N(v)$ for every $y\in V(C)$. In particular, if $|C|\geq 2$, we found a triangle. 

Color $v$ with color $1$, and give each vertex in a singleton component of $Z$ color $1$. For each non-singleton component $C$ of $Z$, note that if $C$ is not bipartite, then $G$ is not 3-colorable  (and we have outcome~\eqref{it:not3}). So assume $C$ is bipartite, and color all vertices from one partition class with $1$. Call~$G'$ the subgraph of $G$ that contains all yet uncolored vertices. (So all remaining vertices of $Z$ from singleton components of $G'-N(v)$.) 

If $G'$ has no edges, we can color $V(G')$ with color $2$ to obtain a valid $2$-coloring of $G$, and are done with outcome~\eqref{it:ccc}.
If $G'$ is bipartite and has an edge $xy$, then we can color $V(G')$ with colors $2$ and $3$ to obtain a valid $3$-coloring of $G$. Observe that if $x,y\in N(v)$, then $G$ has a triangle, and that otherwise, we can assume $x\in N(v)$ and $y\in Z$. In $G$, vertex $y$ belongs to a non-trivial component of $G- N(v)$; thus, as noted above, $G$ has a triangle containing $xy$. In either case, we have outcome~\eqref{it:ddd}.

Now assume $G'$ is not bipartite, that is, $G'$ has an odd cycle  $C_\ell$, on vertices $c_1,\ldots, c_\ell$, say. Then,  for each $c_i$ lying in $Z$, we know that in $G$, there is a vertex $c'_i$ (from the non-trivial component of $Z$ that $c_i$ belongs to) which is adjacent to all three of $c_{i-1}, c_i, c_{i+1}$ (mod $\ell$). So in any valid $3$-coloring of $G$, vertices $c_{i-1}$ and $c_{i+1}$ have the same color. Thus we need to use at least three colors on $V(C_\ell)\cap N(v)$, which makes it impossible to color $v$, unless we use a $4$th color, and we can output~(\ref{it:not3}). This proves the result for $t \leq 4$.
 
Now let $t \ge 5$, and assume that the result is true for all smaller
values of $t$.  For every component $C$ of $Z= V(G) \setminus
(\sset{v} \cup N(v))$, there is a vertex $w_C$ in $N(v)$ with
neighbors in $C$.  We apply the induction hypothesis (for $t-1$) to
$G_C:=G[V(C) \cup \sset{w_C}]$. If this subgraph is not 3-colorable,
neither is $G$ (and we have outcome~\eqref{it:not3}).  If there is an
induced $P_{t-1}$ starting at $w_C$, then we can add $v$ to this path
and have found an induced $P_t$ in $G$ starting at $v$, giving
outcome~\eqref{it:path}.  If neither outcome~\eqref{it:not3} nor
outcome~\eqref{it:path} occured in any component, then each component
$C$ of $G[Z]$ (without $w_C$) can be colored with $2(t-1)-5$ colors if
the algorithm detected a triangle in $G_C$, and with $t-3$ colors
otherwise.

If $N(v)$ is a stable set, and no triangle was detected, then we color each component of $G[Z]$ with $t-3$ colors (which can be repeated), and use one more color for $N(v)$, and repeat one of the colors from $Z$ for $v$ to obtain a $(t-2)$-coloring of $G$, obtaining outcome~\eqref{it:ccc}.

Therefore, we may assume that the algorithm detected a triangle in $G[\{v\} \cup N(v)]$ or some $G_C$, and we output this triangle. 
If $G[N(v)]$ is not bipartite, then $G$ is not 3-colorable. 
Otherwise, we color each component of $G[Z]$ with the same at most $2(t-1)-5$ colors, color $N(v)$ with two new colors, and repeat a color from $Z$ for $v$. 
Then, this yields a coloring of $G$ with $2t-5$ colors, and we found a triangle, which is outcome~\eqref{it:ddd}. 
\end{proof}

In the following, we will use a slightly modified version of this lemma:
\begin{cor} \label{cor:lem1} Let $G$ be connected, $v \in V(G)$, and $t \in \mathbb{N}$. Then, there is a polynomial-time algorithm that outputs
\begin{enumerate}[(a)]
\item that $G$ is not 3-colorable, or
\item an induced path $P_t$ starting with vertex $v$, or
\item a $\max\sset{1, t-2}$-coloring of $G- v$, or
\item a $\max\sset{2, 2t-5}$-coloring of $G- v$ and a triangle in $G$.
\end{enumerate}
\end{cor}
\begin{proof} This is a direct consequence of Lemma~\ref{lem:one} unless $t \leq 3$. 
If $t \leq 3$, then we can find an induced $P_t$ starting at $v$ unless $v$ is adjacent to every vertex in $G - v$. 
So assume $v$ is adjacent to every other vertex. If $G - v$ is not bipartite, then $G$ is not 3-colorable. Otherwise, $G - v$ is $2$-colorable and the algorithm detects a triangle, or $G-v$ is $1$-colorable. 
\end{proof}

\begin{lemma}\label{lem:rt1}
The algorithm from Lemma~\ref{lem:one} (and from Corollary~\ref{cor:lem1}) can be implemented with a running time of $O(t|E(G)|)$ for a connected input graph $G$.
\end{lemma}
\begin{proof}
For $t\leq 4$, we can compute $v, N(v)$ and $Z = V(G) \setminus (\sset{v} \cup N(v))$ in time $\mathcal{O}(|E(G)|)$.  The components of $Z$ can be found in linear time. By going through each vertex $w$ in $N(v)$, and for each such $x$, going through each component $C$ of $Z$ following a connected enumeration of $V(C)$, we can check that $w$ has exactly $0$ or $|V(C)|$ neighbors; if this is not true for some component $C$, then we have found a $P_t$ starting at $v$, obtaining outcome~(\ref{it:path}).

Otherwise, color $v$ with color 1, as well as all components in $Z$ of size 1. If a component $C$ contains two or more vertices, then we check if it is bipartite (in linear time); if not, then since there is a neighbor $w$ of $C$ in $N(v)$ and $w$ is complete to $C$, we output that $G$ is not 3-colorable for outcome~(\ref{it:not3}). If $C$ is bipartite, we choose one of the partition classes of the bipartition, and give all vertices in this class color $1$. 

Let $G'$ be the remaining graph after removing all vertices colored so far. We check if $G'$ has an edge; if not, then we can give a $2$-coloring of $G$  and output~(\ref{it:ccc}). If $G'$ has an edge $xy$, then check if $G'$
 is bipartite. If so, we can get a valid 3-coloring of $G$. Moreover, $xy$ lies in a triangle (either because $x,y\in N(v)$ or because $x$ and $y$ have a common neighbour in $Z$), and we can output~(\ref{it:ddd}). So assume we found that $G'$ is not bipartite, that is we found an odd cycle $C_\ell$ in $G'$, on vertices $c_1,\ldots, c_\ell$, say.
For each $c_i\in Z\cap V(G')$, there is a vertex $c'_i\in Z\cap V(G)$ adjacent to all three of $c_{i-1}, c_i, c_{i+1}$ (mod $\ell$), hence we can output~(\ref{it:not3}), as vertices $c'_i$, $V(C_\ell)$ and $v$ induce an obstruction to $3$-coloring $G$.

Now let $t \geq 5$. We compute $N(v)$ in time $|d(v)|$, compute components of $G - (\sset{v} \cup N(v))$ in linear time, check if $N(v)$ is bipartite in linear time (if not, return that $G$ is not 3-colorable), check if $N(v)$ contains two adjacent vertices, and correspondingly 1 or 2-color $N(v)$. Then we go through $N(v)$ to find a neighbor $w_C$ for each component $C$ of $G - (\sset{v} \cup N(v))$ and run the algorithm with vertex $w_C$ and parameter $t-1$ on the component $C$. 

If the outcome in any component $C$ is an induced $P_{t-1}$ starting at $w_c$, we can add $v$ at the start of the path and get outcome~(\ref{it:path}). If some component is not 3-colorable, then neither is $G$, giving outcome~\eqref{it:not3}. Otherwise, we find the necessary colorings (and possibly a triangle) to output~(\ref{it:ccc}) or~(\ref{it:ddd}).

Note that no edge occurs in two components, therefore we require $O(|E(G)|)$ processing time before using recursion and a total amortized running time of at most $O((t-1)|E(G)|)$ for recursive calls of the algorithm, which implies the overall running time. 
\end{proof}

Let $S$ be a set of vertices of $G$, then we let $F(S)$ denote the smallest set so that $S \subseteq F(S)$ and no vertex in $G - F(S)$ has two adjacent neighbors in $F(S)$. $F(S)$ can be computed by repeatedly adding vertices that have two adjacent neighbors in the current set. In a 3-coloring, the colors of the vertices in $S$ uniquely determine the colors of all vertices in $F(S)$. 

\begin{lemma} \label{lem:main}
Let $G$ be connected, $v \in V(G)$ and $k,t\in \mathbb{N}$. There is a polynomial-time algorithm that outputs 
\begin{enumerate}[(a)]
\item \label{it:not3c} that $G$ is not 3-colorable, or
\item \label{it:pt} an induced path $P_t$ in $G$, or
\item \label{it:pk} an induced path $P_k$ in $G$ starting in $v$, or
\item a set $S$ of size $\max\sset{1,k-2}$ with $v \in S$, and a $\max\sset{1,\ceil{\frac{t-1}{2}}-2}$-coloring of $G - (F(S) \cup N(F(S)))$, or 
\item a set $S$ of size $\max\sset{1,k-2}$ with $v \in S$, and a $\max\sset{2,2\ceil{\frac{t-1}{2}}-5}$-coloring of $G - (F(S) \cup N(F(S)))$, and a triangle in $G$.
\end{enumerate}
\begin{proof}
We prove this by induction on $k$. If $k \leq 3$, then this follows from Corollary~\ref{cor:lem1} with input $k$ and vertex $v$, by setting $S = \sset{v}$ and noting that then $F(S)=S$. 

Now let $k > 3$. Note that we can assume $k \leq t$, because otherwise we can run the algorithm for $k$ set to~$t$, and all outcomes except (c) will be valid for the original $k$ as well, and if we do get outcome (c), we can use it as outcome (b) instead. Furthermore, if $3 \leq k \leq \ceil{\frac{t-1}{2}}$, then the result follows from Corollary~\ref{cor:lem1} with input $k$ and vertex $v$, and setting  $S = \sset{v}$. 

Consider $Z = V(G) \setminus (N(v) \cup \sset{v})$. Let $\mathcal{C} = \sset{C_1, \dots, C_r}$ be the list of components of $G[Z]$, and let $\mathcal{D} = \sset{D_1, \dots, D_l}$ be the list of components of $G[N(v)]$. We now describe a procedure where at each step, we color one of the components of  $\mathcal C$, and then put it aside, to go on working with the remaining graph, until one component $D \in \mathcal{D}$ has neighbors in all remaining components of $\mathcal{C}$. 

The details are as follows. While there is not a single component in $\mathcal{D}$ with neighbors in every component of $\mathcal{C}$, let $D, D' \in \mathcal{D}$, $C, C' \in \mathcal{C}$ so that $C$ has a neighbor $x$ in $D$ but no neighbor in $D'$, and $C'$ has a neighbor $x'$ in $D'$ but no neighbor in $D$. To see this, choose a components $D \in \mathcal{D}$ that has neighbors in as many components of $\mathcal{C}$ as possible. If some $C' \in \mathcal{C}$ has no neighbor in $D$, then there is a component $D' \in \mathcal{D}$ with a neighbor in $C'$ by connectivity. But $C'$ has a neighbor in $D'$ but not $D$, so by choice of $D$, $D'$ cannot have a neighbor in all components $C \in \mathcal{C}$ in which $D$ has a neighbor; thus let $C$ be a component in $\mathcal{C}$ so that $D$ has a neighbor in $C$ and $D'$ does not. These are the desired components. 
 
Then, we apply Corollary~\ref{cor:lem1} to $\sset{x} \cup C$ (with parameter $\ceil{\frac{t-1}{2}}$ and vertex $x$) and to $\sset{x'} \cup C'$ (with parameter $\ceil{\frac{t-1}{2}}$ and vertex $x'$). If either of these graphs is not 3-colorable, then $G$ is not 3-colorable. If,  in both cases, there is an induced $P_{\ceil{\frac{t-1}{2}}}$ starting at $x$ and at $x'$, respectively, then, since $x$ has no neighbors in $C' \cup D'$ and $x'$ has no neighbors in $D \cup C$, we can combine them, using the path $xvx'$, to obtain  an induced $P_{2\ceil{\frac{t-1}{2}} + 1}$ in $G$, which contains an induced $P_t$. Thus, we can assume that for at least one of the two components, we found a coloring instead. In particular, we found a coloring of $C$ or of $C'$ with $\max\sset{1, \lceil\frac{t-1}{2}\rceil - 2}$ colors, or a triangle in $G$, and a coloring  of $C$ or of $C'$ with $\max\sset{2, 2\ceil{\frac{t-1}{2}}-5}$ colors. We then remove the component with the coloring and continue. 

Finally, we arrive at a point where there is a component $D \in \mathcal{D}$ that has neighbors in all remaining components of $\mathcal{C}$. Note that if $S$ includes $v$ and any vertex $x \in D$, then $F(S) \supseteq D$ and thus $N(D) \subseteq F(S) \cup N(F(S))$. Therefore, we call a remaining component of $\mathcal{C}$ \emph{good} if it is contained in $N(D)$, and \emph{bad} otherwise. Our goal is to find a vertex $x \in D$ with neighbors in all bad components. 

While there is no vertex in $D$ with neighbors in all bad components,
we can find two bad components $C,C'$ among the remaining components
of $\mathcal{C}$ such that $C$ has a neighbor $y$ in $D$, $C'$ has a
neighbor $y'$ in $D$, $y$ has no neighbors in $C'$ and $y'$ has no
neighbors in $C$. As before, we can find these components by choosing $y$ with neighbors in as many bad components as possible, and then letting $y'$ be a vertex with a neighbor in a bad component $C'$ in which $y$ does not have a neighbor. Consequently, $y'$ has no neighbor in at least one bad component $C$ in which $y$ does have a neighbor. 

As $C$ and $C'$ are bad, there exist components $E$ and $E'$, of $C
\setminus N(D)$, and of $C' \setminus N(D)$, respectively.  Let $x$ be
the first vertex on a shortest path $P$ from $E$ to $y$, and define
$x'$ and $P'$ analogously. Apply Corollary~\ref{cor:lem1} to
$G[\sset{x} \cup E]$ (with parameter $\ceil{\frac{t-2}{2}}$ and vertex
$x$) and to $G[\sset{x'} \cup E']$ (with parameter
$\ceil{\frac{t-2}{2}}$ and vertex~$x'$). If either of these two graphs
is not 3-colorable, then $G$ is not 3-colorable. If, in both cases,
there is an induced $P_{\ceil{\frac{t-2}{2}}}$, say $P$ starting at $x$ and $P'$ starting at
$x'$, respectively, then we can combine these paths to an induced
path of length at least $t$ by taking $xPyy'P'x'$ or
$xPyvy'P'x'$ (depending on whether $yy'$ is an edge or
not).  Thus, we can assume that for at least one of $G[\sset{x} \cup
E]$, $G[\sset{x'} \cup E']$, we found a coloring instead. In
particular, we found a coloring of $E$ or of $E'$ with $\max\sset{1,
  \lceil\frac{t-2}{2}\rceil - 2}$ colors, or a triangle in $G$, and a
coloring with $\max\sset{2, 2\ceil{\frac{t-2}{2}}-5}$ colors.  We then
remove the component with the coloring and continue.

When this terminates, there is a single vertex $v' \in D$ that has neighbors in all remaining components of $\mathcal{C}$, except possibly those contained in $N(D)$. Let $V'$ be the set of vertices in those components. Then we can apply the induction hypothesis with $k-1$, $t$, and $v'$ to $G[V' \cup \sset{v'}]$. If this graph is not 3-colorable, neither is $G$. If it contains an induced path $P_t$, so does $G$. If it contains an induced path $P_{k-1}$ starting in $v'$, then we can add $v$ to this path to obtain an induced path $P_k$ starting in $v$. If there is a set $S$ of size $k-3$ with $v' \in S$, and a $\max\sset{1,\ceil{\frac{t-1}{2}}-2}$-coloring of $G - (F(S) \cup N(F(S)))$ or a $\max\sset{2,2\ceil{\frac{t-1}{2}}-5}$-coloring of $G - (F(S) \cup N(F(S)))$ and a triangle, then we proceed as follows.
We add $v$ to $S$, and now $S$ has size $k-2$. Moreover, since both $v$ and $v'$ in $S$, we know that $D \in F(S)$, thus all vertices in $\sset{v} \cup N(v) \cup D \cup N(D)$ are in $F(S) \cup N(F(S))$. We colored different components of $Z \setminus N(D)$ at different stages, but we can reuse the colors used on these components. Therefore, this leads to outcome (d) or~(e), depending on if the algorithm detected a triangle at any stage. 
\end{proof}
\end{lemma}

\begin{lemma} \label{lem:rtmain} The algorithm from Lemma~\ref{lem:main} can be implemented with running time $O\left(k\ceil{\frac{t-1}{2}}|E(G)|\right)$ for a connected input graph $G$. 
\end{lemma}
\begin{proof} For $k \leq 3$, this follows from Lemma~\ref{lem:rt1}. For $k > 3$, we compute $Z = V(G) \setminus (N(v) \cup \sset{v})$ and the components $C_1, \dots C_r$ of $G[Z]$ and the components $D_1, \dots, D_r$ of $G[N(v)]$ as in the proof of Lemma~\ref{lem:main}; all this can be done in linear time. 

Now, subsequently, for $j = 1, \dots, r$, we consider $D_j$, and some of the $C_i$ adjacent to it, and after possibly coloring and deleting these components $C_i$, we might also delete $D_j$. More precisely, for $j = 1, \dots, r$, 
 we consider those $C_i$ that only have neighbors in $D_j$. For each such $C_i$, choose a neighbor  $x_{ij}$ in $D_j$ and apply Corollary~\ref{cor:lem1} to $G[C_i \cup \sset{x_{ij}}]$. If this graph is not 3-colorable, then neither is $G$. If the algorithm returns a coloring (and possibly a triangle), then this is the coloring we will use in $C_i$, as explained in the proof of Lemma~\ref{lem:main}, so we can delete $C_i$. (But, if the algorithm found a triangle, we shall remember this triangle for a possible output, at least if it is the first one to be found.) Otherwise, the algorithm returns a path of length $\ceil{\frac{t-1}{2}}$, which we keep. After going through all $C_i$ with neighbors only in $D_j$, we delete $D_j$ if the algorithm always returned a coloring (or if there were no $C_i$ to consider). 
That is, we keep $D_j$ if and only if for some $i$ we found a path starting from $x_{ij}$ with interior in $C_i$. 

The amortized time it takes to process all $D_j$ is $O(\ceil{\frac{t-1}{2}}|E(G)|).$ This is so because every component~$C_i$ is used for the algorithm from Corollary~\ref{cor:lem1} at most once. 

In the end, if there are $D_j$, $D_{j'}$ that we did not delete, then there is a component $C_i$  that only has neighbors in $D_j$, and another component $C_{i'}$ that only  has  neighbors in $D_{j'}$, and both $C_i$ and $C_{i'}$ contain a $P_{\ceil{\frac{t-1}{2}}}$ in their interior, starting at $x_{ij}$ and $x_{i'j'}$ respectively. By connecting them using the middle segment $x_{ij} v x_{i'j'}$, we find a $P_t$ in $G$ that we can output as outcome~(\ref{it:pt}). 
Otherwise, there is only one $D_j$ left at the end. Since  whenever we deleted a component $D_{j'}$, we  ensured that each remaining~$C_i$ has a neighbor in some $D_{j}$ with $j \neq j'$, this means that $D_j$ has neighbors in all $C_i$ that we did not color yet. 

Let $Z'$ be the set of vertices of remaining components $C_i$, and let $Z'' = Z' \setminus N(D)$. 
Each component of $G[Z'']$ is contained in some component $C_i$ and thus it has a neighbor in $N(D)$. Therefore, we can apply the same argument as before to components of $G[Z'']$ and components of $N(D)$ with neighbors in them. Whenever the algorithm for Corollary~\ref{cor:lem1} outputs a coloring we keep it (and we also keep the possibly found triangle, if it is the first triangle to be found), and if it outputs that a component is not 3-colorable, then $G$ is not 3-colorable, and if there is a path $P_{\ceil{\frac{t-2}{2}}}$, we keep track of it. (Note that components of $N(D)$ that are not adjacent to $Z''$ get deleted automatically.) When this terminates, if there are still two components of $N(D)$, then there are two paths we can combine to a $P_t$ as in Lemma~\ref{lem:main}. Otherwise, a single component $D^*$ of $N(D)$ has a neighbor in all remaining components $C_1', \dots, C_s'$, so there is a vertex $v' \in D$ so that $\sset{v'} \cup D^* \cup C_1' \cup \dots \cup C_s'$ is connected, and $v'$ is the only neighbor of $v$ in that subgraph. Next, we apply induction for $k-1$ with root vertex $v$ on that set. 
If this finds a set $S$ and a coloring, we add $v$ to $S$. 
If it finds a path $P_{k-1}$, we add $v$ to the path. 
If it finds a $P_t$, we output it. 
If it is not 3-colorable, then neither is $G$. 

The total running time of the recursive application of the algorithm is $O\left((k-1)\ceil{\frac{t-1}{2}}|E(G)|\right)$, and all preprocessing steps leading there can be implemented with a running time of $O\left(\ceil{\frac{t-1}{2}}|E(G)|\right)$, which implies the result. 
\end{proof}

We can now give the proof of our main result.

\begin{proof}[Proof of Theorem~\ref{thm:main}]
We use the algorithm from Lemma~\ref{lem:main} with $k = t$. Since only outcomes (d) and (e) can occur in this setting, it is sufficient to show that if $G$ is 3-colorable, we can find a 3-coloring of $F(S) \cup N(F(S))$ in time $O(3^{t-2} \cdot \textnormal{poly}(|V(G)|))$, as follows: For each vertex in the set $S$, we try each possible color for a total of at most $3^{t-2}$ possibilities. By definition of $F(S)$, this determines the color of each vertex in $F(S)$, and for each vertex in $N(F(S))$, there are at most two possible colors. Thus, we reduced to a 2-list-coloring problem, which can be solved in linear time by reduction to~\textsc{2Sat}~\cite{edwards, erdos, vizing}. Hence if $G$ is 3-colorable, we can 3-color $F(S) \cup N(F(S))$, and add these three new colors to the coloring from Lemma~\ref{lem:main}.

The total running time follows from the running time of the algorithm in Lemma~\ref{lem:rtmain} in addition to an algorithm determining the connected components of $G$. 
\end{proof}

By combining Theorem~\ref{thm:main} with Lemma~\ref{lem:one}, we obtain the algorithms for coloring $3$-colorable $P_t$-free graphs with the number of colors shown in Table~\ref{t:k3} (if there is a triangle) and Table~\ref{t:no-k3} (if there is no triangle). For $t$ larger than shown in the table, Theorem~\ref{thm:main} uses a smaller number of colors. 

\begin{table}[htb]
\begin{center}
\begin{tabular}{l|r|r|r|r|r|r|r|r|r|r }
 $t$ & 3 & 4 & 5 & 6 & 7 & 8 & 9 & 10 & 11 & $> 11$ \\  
\hline 
 $\max\sset{3, 2t-5}$ & 3 & 3 & 5 & 7 & 9 & 11 & 13 & 15 & 17 &\\   
 $\max\sset{5, 2\ceil{\frac{t-1}{2}}-2}$ & 5 & 5 & 5 & 5 & 5 & 6 & 6 & 8 & 8 &\\   
 Best option & 3 & 3 & 3\footnotemark & 5 & 5 & 6 & 6 & 8 & 8 &$2\ceil{\frac{t-1}{2}}-2$\\   
\end{tabular}
\end{center}
\caption{Number of colors we use for a 3-colorable $P_t$-free graph if there is a triangle}
\label{t:k3}
\end{table}
\footnotetext{
If $t = 5$, we can improve the number of colors required if there is a triangle to 3, because it cannot happen that there are two components $C, C'$ of $G - (\sset{v} \cup N(v))$ and components $D, D'$ of $G[N(v)]$ so that $C$ has a neighbor in $D$ but not $D'$, and $C'$ has a neighbor in $D'$ but not $D'$, because this already yields an induced $P_5$. Thus, by induction, all vertices will be in $F(S) \cup N(F(S))$, where we can test for 3-colorability as described in Theorem~\ref{thm:main}. }

\begin{table}[htb]
\begin{center}
\begin{tabular}{l|r|r|r|r|r|r|r|r|r|r }
 $t$ & 3 & 4 & 5 & 6 & 7 & 8 & 9 & 10 & 11 & $> 11$\\  
\hline 
 $\max\sset{2, t-2}$ & 2 & 2 & 3 & 4 & 5 & 6 & 7 & 8 & 9& \\   
 $\max\sset{4, \ceil{\frac{t-1}{2}}+1}$ & 4 & 4 & 4 & 4 & 4 & 5 & 5 & 6 & 6 &\\   
 Best option & 2 & 2 & 3 & 4 & 4 & 5 & 5 & 6 & 6 & $\ceil{\frac{t-1}{2}} + 1$\\   
\end{tabular}
\end{center}
\caption{Number of colors we use for a 3-colorable $P_t$-free graph if there is no triangle}
\label{t:no-k3}
\end{table}

\section{Hardness result} \label{sec:hard}

In this section, we show that improving Lemma~\ref{lem:one} is hard. More precisely: 

\begin{theorem}\label{lem:hardness}
Let $G$ be a connected graph and $v \in V(G)$ so that there is no induced $P_t$ in $G$ starting at $v$. Then, deciding $k$-colorability on this class of graphs is NP-hard if $k \geq 4$ and $t \geq 3$ or if $k = 3$ and $t \geq 5$. It can be solved in polynomial time if $t \leq 2$ or if $k = 3$ and $t \leq 4$. 
\end{theorem}

\begin{proof}
For the polynomial time solvability, observe that 
if $t \leq 2$, then $|V(G)| \leq 1$. If $k = 3, t \leq 4$, then the result follows from Lemma~\ref{lem:one}. 

For the hardness, first consider the case $k \geq 4$, $t \geq 3$. In this case, we can reduce the 3-coloring problem to this problem by taking any instance $G$ and adding a clique of size $k-3$ complete to $G$. Then, no vertex in this clique starts a $P_3$, but the resulting graph is $k$-colorable if and only if $G$ is $3$-colorable. 

It remains to consider the case $k = 3$, $t\geq 5$. We show a reduction from the NP-complete problem \textsc{NAE-3Sat} \cite{garey}. An instance of \textsc{NAE-3Sat} is a boolean formula with variables $x_1, \dots, x_n$ and clauses $C_1, \dots, C_m$, where each clause contains exactly three literals (variables or their negations). It is a \textsc{Yes}-instance if and only if there is an assignment of the variables as true or false so that for every clause, not all three literals in the clause are true, and not all three are false. 

We construct a graph $G$ as follows: $G$ contains a vertex $v$, vertices labeled $x_i$ and $\overline{x_i}$ for $1 \leq i \leq n$, and a triangle $T_j$ for each clause $C_j$. The vertex $v$ is adjacent to all vertices $x_i$ and $\overline{x_i}$, but not to any of the triangles. For each $i$, $x_i$ is adjacent to $\overline{x_i}$. For each clause $C_j$, we assign each literal a vertex of the triangle $T_j$, and connect this vertex to the literal (the vertex labeled $x_i$ or $\overline{x_i}$). There are no other edges in $G$. 

Then, there is no $P_5$ starting at $v$ in $G$, because such a path would have to contain exactly one of the vertices labeled $x_i$ and $\overline{x_i}$, and this would be the second vertex of the path. As there are no edges between the triangles $T_j$, all remaining vertices of the $P_5$ would have to be in one triangle $T_j$. But no triangle can contain a $P_3$. Therefore, $G$ is a valid instance. 

It remains to show that $G$ is 3-colorable if and only if the instance of \textsc{NAE-3Sat} is a \textsc{Yes}-instance. If $G$ has a 3-coloring, then the neighbors of $v$ are 2-colored (say with colors $1$ and $2$) and $x_i$ never receives the same color as $\overline{x_i}$. Assign the variables so that literals colored 1 are true, and those colored 2 are false. Then, if there is a clause $C_j$ so that all of its literals are true, this means that each vertex of $T_j$ has a neighbor colored 1, so $T_j$ uses only colors $2$ and $3$, which is impossible in a valid coloring of a triangle. For the same reason, there cannot be a clause so that all of its literals are false. Thus,~$G$ was constructed from a \textsc{Yes}-instance. 

Conversely, if the instance we started with is a \textsc{Yes}-instance, we color $v$ with color 3, true literals with color 1, and false literals with color 2. For each triangle $T_j$, one of the vertices adjacent to a true literal is colored 2, one of the vertices adjacent to a false literal is colored 1, and the remaining vertex is colored 3. This is a valid 3-coloring of $G$. 
\end{proof}

\section{Conclusion}

In this paper we showed how to color a given 3-colorable $P_t$-free graph with a number of colors that is $t$, roughly.
The running time of our algorithm is of the form $O(f(t) \cdot n^{O(1)})$, when the input graph has $n$ vertices, and thus FPT in the parameter $t$. (The class FPT contains the fixed parameter tractable problems, which are those that can be solved in time $f(k) \cdot {|x|}^{O(1)}$ for some computable function $f$.)

In view of this, it seems to be an intriguing question whether the 3-coloring problem is fixed-parameter tractable when parameterized by the length of the longest induced path.
That is, whether there is an algorithm with running time $O(f(t) \cdot n^{O(1)})$ that decides 3-colorability in $P_t$-free graphs.
So far, however, it is not even known whether there is an XP algorithm to decide 3-colorability in $P_t$-free graphs. (XP is the class of parameterized problems that can be solved in time $O(n^{f(k)})$ for some computable function $f$.)
If such an XP-algorithm existed, this would show that the problem is in P whenever $t$ is fixed.
Therefore, attempting to prove W[1]-hardness seems to be more reasonable than trying to prove that the problem is in FPT.

Another question we addressed is $k$-coloring connected graphs so that some vertex is not the end vertex of an induced $P_t$. We showed that coloring in this case is NP-hard whenever $k=3$ and $t \geq 5$, or $k \geq 4$ and $t \geq 3$. Lemma~\ref{lem:one} gives a simple algorithm for an $f(t)$-approximate coloring for $k=3$ and any $t$, and it would be interesting to have a complementing result proving hardness of approximation. On the other hand, any improvement of Lemma~\ref{lem:one} would immediately yield an improvement of our main result. 

\section*{Acknowledgments}

We are thankful to Paul Seymour for many helpful discussions. We  thank Stefan Hougardy for pointing out \cite{kawa} to us.

\end{document}